\documentclass{amsart}
\usepackage{amssymb}
\usepackage{pdflscape}
\usepackage{amscd}
\usepackage{fancyhdr}

\newtheorem{theorem}{Theorem}[section]
\newtheorem{corollary}[theorem]{Corollary}
\newtheorem{lemma}[theorem]{Lemma}
\newtheorem{proposition}[theorem]{Proposition}
\newtheorem{conjecture}[theorem]{Conjecture}
\newtheorem{example}[theorem]{Example}

\newtheorem{definition}[theorem]{Definition}

\title{Punctual noncommutative Hilbert schemes}
\author{Markus Reineke}
\address{Markus Reineke,
Ruhr-Universit\"at Bochum, Faculty of Mathematics, Universit\"atsstra{\ss}e 150, 44780 Bochum, Germany}
\email{Markus.Reineke@ruhr-uni-bochum.de}

\begin{document}
\begin{abstract} Punctual noncommutative Hilbert schemes are projective varieties parametrizing finite codimensional left ideals in noncommutative formal power series rings. We determine their motives and intersection cohomology, by constructing affine pavings and small resolutions of singularities.\end{abstract}
\maketitle
\parindent0pt

\section{introduction}\label{intro}

Noncommutative Hilbert schemes, which are varieties parametrizing finite codimensional left ideals in free algebras, were first defined in \cite{Nori}, appeared as generic Brauer-Severi schemes in \cite{lebruyn,vandenbergh}, and were shown to admit a natural affine paving indexed by trees in \cite{R}. Subsequently, they (and their generalizations to arbitrary quivers \cite{ER,functional}) played a role in motivic Donaldson-Thomas theory, see for example \cite{franzen,joycesong,reinekedocumenta}.\\[1ex]
In the same way as punctual Hilbert schemes appear naturally in the study of Hilbert schemes of points \cite{iarrobino}, it is natural to study a punctual analogue of noncommutative Hilbert schemes  parametrizing finite codimensional ideals in noncommutative formal power series rings. These varieties  were studied in \cite{lebruyn} as particular fibres of generic Brauer-Severi schemes, whose equidimensionality is proved there.\\[1ex]
In the present work, after summarizing results of \cite{R} in Section \ref{nchilb}, we define punctual noncommutative Hilbert schemes in Section \ref{def}, give an invariant-theoretic interpretation, construct an embedding into a Grassmannian, and provide several small examples.\\[1ex]
Using a Harder-Narasimhan type stratification, we describe the generating series of motives of these varieties in Section \ref{motive1} as the solution to an algebraic functional equation, closely related to similar results in \cite{R} (see also \cite{ACMR}). Using this description of the motives at a key point, we show in Section \ref{paving} that the affine pavings of noncommutative Hilbert schemes constructed in \cite{R} in fact restrict to affine pavings of their punctual analogues.\\[1ex]The punctual noncommutative Hilbert schemes typically being singular, it is desirable to construct resolutions of singularities. This is accomplished in Section \ref{resolution} by naturally generalizing the Springer resolutions of nullcones. Again, we can compute the motives of the resulting smooth varieties, finally resulting in a very simple formula in Section \ref{motive2}. Very surprisingly, in Section \ref{small} our resolution turns out to be small, as a dimension estimate for a Steinberg-type variety shows. We can thus conclude with a closed formula for the Poincar\'e polynomial in intersection homology of the punctual noncommutative Hilbert schemes.\\[1ex]
The main results can be summarized as follows:
\begin{theorem} Let ${}^0{\rm Hilb}^{(m)}(\mathbb{C}^d)$ be the Hilbert scheme parametrizing codimension $d$ left ideals in $\mathbb{C}\langle\langle x_1,\ldots,x_m\rangle\rangle$. It is an irreducible projective variety of dimension $$\dim{}^0{\rm Hilb}^{(m)}(\mathbb{C}^d)=(m-1)d(d-1)/2$$ admitting an affine paving and a small resolution of singularities. The generating function of motives
$${}^0F^{(m)}(t)=\sum_{d\geq 0}\mathbb{L}^{-(m-1)d(d-1)/2}\cdot[{}^0{\rm Hilb}^{(m)}(\mathbb{C}^d)]t^d\in 1+tK_0({\rm Var}_\mathbb{C})[\mathbb{L}^{-1}][[t]]$$
is uniquely determined by
$${}^0F^{(m)}(t)=1+t\cdot\prod_{k=1}^m{}^0F^{(m)}(\mathbb{L}^{1-k}t).$$
Its Poincar\'e polynomial in rational intersection homology is given by
$$\sum_i\dim{\rm IH}^i({}^0{\rm Hilb}^{(m)}(V),\mathbb{Q})q^{i/2}=\prod_{i=0}^{d-1}\frac{q^{(m-1)i+1}-1}{q-1}.$$
\end{theorem}

{\bf Acknowledgments:} The author would like to thank Ben Davison, Hans Franzen and Lydia Gösmann for inspiring discussions leading to the present research, and especially Pieter Belmans for key help in identifying a particular punctual Hilbert scheme in Example \ref{example}. The author is grateful to the MFO Oberwolfach and to the organizers of the CARE conference at ENS Lyon, where part of this research was carried out, for excellent working conditions.

\section{Recollections on noncommutative Hilbert schemes}\label{nchilb}

For the following material, we refer to \cite{R}. Fix $m\geq 1$. In the following, $V$ will always denote a complex vector space of dimension $d\geq 0$. The group ${\rm GL}(V)$ acts on ${\rm End}(V)^m$ by simultaneous conjugation. We denote by
$$X^{(m)}(V)={\rm End}(V)^m//{\rm GL}(V)$$
the invariant-theoretic quotient, that is, the spectrum of the ring of invariants
$$R^{(m)}(V)=\mathbb{C}[{\rm End}(V)^m]^{\rm GL(V)},$$
which is generated by functions
$$t_\omega(\varphi_1,\ldots,\varphi_m)={\rm tr}(\varphi_{i_s}\circ\ldots\circ\varphi_{i_1})$$
for words $\omega=(i_1\ldots,i_s)$ of length $s\geq 0$ in the alphabet $\{1,\ldots,m\}$.
The points of $X^{(m)}(V)$ naturally correspond to isomorphism classes of semisimple representations of the free algebra $A^{(m)}=\mathbb{C}\langle x_1,\ldots,x_m\rangle$ on $V$. The variety $X^{(m)}(V)$ is irreducible and affine, of dimension $$\dim X^{(m)}(V)=(m-1)d^2+1$$
if $m\geq 2$, and isomorphic to $\mathbb{C}^d$ in case $m=1$. The dilation action of $\mathbb{C}^*$ on ${\rm End}(V)^m$ being compatible with the ${\rm GL}(V)$-action, it induces an action on $X^{(m)}(V)$, which turns the latter into a cone with vertex $0$, the point corresponding to the zero orbit. \\[1ex]
We consider the ${\rm GL}(V)$-representation ${\rm End}(V)^m\times V$ and always denote its points by
$$(\varphi_*,v)=(\varphi_1,\ldots,\varphi_m,v)$$
for linear operators $\varphi_k$ and a vector $v$. We consider the open subset $({\rm End}(V)^m\times V)_{\rm st}$ of stable points, defined by the condition that 
$$\mathbb{C}\langle\varphi_1,\ldots,\varphi_m\rangle\cdot v=V,$$
that is, $v$ is a cyclic vector for the representation of $A^{(m)}$ on $V$ defined by the $\varphi_k$. The natural ${\rm GL}(V)$-action on this subset admits a geometric quotient
$${\rm Hilb}^{(m)}(V)=({\rm End}(V)^m\times V)_{\rm st}/{\rm GL}(V),$$
called the noncommutative Hilbert scheme. It is a smooth irreducible quasiprojective variety of dimension $(m-1)d^2+d$.\\[1ex]
Its points parametrize $d$-codimensional left ideals in the free algebra $A^{(m)}$. Namely, the annihilator
$${\rm Ann}(v)=\{P(x_1,\ldots,x_m)\in A^{(m)}\,:\, P(\varphi_*)v=0\}$$
(for $P(\varphi_*)=P(\varphi_1,\ldots,\varphi_m)$)
is such an ideal and, conversely, such an ideal $I\subset A^{(m)}$ gives rise to the (well-defined up to change of basis) $d$-dimensional space $A^{(m)}/I$ with the $m$ linear operators of multiplication by the $x_k$ and the cyclic vector $1+I$.\\[1ex]
The variety ${\rm Hilb}^{(m)}(V)$ can be realized invariant-theoretically as the ${\rm Proj}$ of the graded ring $\widehat{R}^{(m)}(V)$ of semi-invariant functions on ${\rm End}(V)^m\times V$ with respect to some multiple of the determinant character; this ring is generated over $R^{(m)}(V)$ by determinant functions $$D_{P_1,\ldots,P_d}(\varphi_*,v)=\det[P_1(\varphi_*)v|\ldots  P_d(\varphi_*)v]$$ for $d$-tuples $(P_1,\ldots,P_d)$ of polynomials in $A^{(m)}$.\\[1ex]
The natural map from the ${\rm Proj}$ of a graded ring to the ${\rm Spec}$ of its degree zero part yields a projective map
$$\pi:{\rm Hilb}^{(m)}(V)\rightarrow X^{(m)}(V)$$
induced on points by forgetting the cyclic vector (this can be viewed as a noncommutative analogue of the Hilbert-Chow morphism from a Hilbert scheme of points in a variety to the corresponding symmetric product). Viewing a point in ${\rm Hilb}^{(m)}(V)$ as an ideal $I$ as above, $\pi(I)$ corresponds to the semisimplification of the representation $A^{(m)}/I$.\\[1ex]
Let $\Omega$ be the set of finite words in the alphabet $\{1,\ldots,m\}$ (which we visualize as a free $m$-ary tree with the empty word as its root); we totally order $\Omega$ by the lexicographic order on words induced by the total order $1<2<\ldots<m$. A subset $T\subset\Omega$ is called a tree if it is closed under taking left subwords.\\[1ex]
For a $d$-element tree $T$, we choose a basis $(e_\omega)$ of $V$ indexed by the words $\omega\in T$.  We define a subset $\widetilde{S}_T\subset{\rm End}(V)^m\times V$ as the set of tuples $(\varphi_*,v)$ such that
\begin{enumerate}
\item $v=e_\emptyset$,
\item $\varphi_k(e_\omega)=e_{\omega k}$ if $\omega,\omega k\in T$,
\item $\varphi_k(e_\omega)\in\langle e_{\omega'}\, :\, \omega'\in T,\, \omega'<_{\rm lex}\omega k\}$ if $\omega\in T$, $\omega k\not\in T$.
\end{enumerate}
We denote by $S_T\subset{\rm Hilb}^{(m)}(V)$ the image of $\widetilde{S}_T$ under the quotient map. Then the $S_T$, for $T$ ranging over the $d$-element trees $T$, form an affine paving of ${\rm Hilb}^{(m)}(V)$. 

We thus see that the motive of ${\rm Hilb}^{(m)}(V)$, that is, its class in the Grothendieck ring $K_0({\rm Var}_\mathbb{C})$ of complex varieties, is a polynomial in the Lefschetz motive $\mathbb{L}$. We can characterize the generating series of all these motives by a simple functional equation:

\begin{theorem}\cite[Theorem 5.5]{R}\label{thmold} The series
$$F^{(m)}(t)=\sum_{d\geq 0}\mathbb{L}^{-(m-1)d(d+1)/2-d}[{\rm Hilb}^{(m)}(\mathbb{C}^d)]t^d\in 1+t K_0({\rm Var}_\mathbb{C})[\mathbb{L}^{-1}][[t]]$$
is uniquely determined by the functional equation
$$F^{(m)}(t)=1+t\cdot\prod_{k=1}^mF^{(m)}(\mathbb{L}^{k-1}t).$$
\end{theorem}

For all properties of motives we will use in the following, we refer to \cite{Br}.

\section{Definition of punctual noncommutative Hilbert schemes}\label{def}

Recall the projective map $$\pi:{\rm Hilb}^{(m)}(V)\rightarrow X^{(m)}(V).$$
Our central object of interest is the most special fibre of this map:

\begin{definition} Define
$$ ^0{\rm Hilb}^{(m)}(V)=\pi^{-1}(0).$$
\end{definition}

${}^0{\rm Hilb}^{(m)}(V)$ is thus a projective variety. By the representation-theoretic description of $\pi$ above, its points can be viewed as $d$-codimensional left ideals in the noncommutative formal power series ring $\widehat{A}^{(m)}=\mathbb{C}\langle\langle x_1,\ldots,x_m\rangle\rangle$, justifying the name punctual noncommutative Hilbert scheme.\\[1ex]
We can interpret $^0{\rm Hilb}^{(m)}(V)$ as a geometric quotient as follows. Let $\mathcal{N}^{(m)}(V)\subset{\rm End}(V)^m$ be the cone of simultaneously nilpotent linear operators, which is a closed irreducible subvariety; it is the zero fibre of the quotient map
$${\rm End}(V)^m\rightarrow X^{(m)}(V)$$
(we refer to \cite{GR} for all required properties of $\mathcal{N}^{(m)}(V)$). We define
$$(\mathcal{N}^{(m)}(V)\times V)_{\rm st}=(\mathcal{N}^{(m)}(V)\times V)\cap({\rm End}(V)^m\times V)_{\rm st}.$$
Then $$^0{\rm Hilb}^{(m)}(V)=(\mathcal{N}^{(m)}(V)\times V)_{\rm st}/{\rm GL}(V).$$ 
We can explicitly coordinatize $^0{\rm Hilb}^{(m)}(V)$ by embedding it into a Grassmannian. Namely, if $\varphi_*\in\mathcal{N}^{(m)}(V)$, then there exist a complete flag $F_*$ in $V$ such that $\varphi_k(F_i)\subset F_{i-1}$ for all $i=1,\ldots,d$ and all $k=1,\ldots,m$. In  particular, every $d$-fold product of the $\varphi_k$ equals zero. The corresponding codimension $d$ left ideal $I\subset \widehat{A}^{(m)}$ thus contains the $d$-th power of the augmentation ideal 
 $\widehat{A}^{(m)}_+=(x_1,\ldots,x_m)$ of $\widehat{A}^{(m)}$, providing us with a well-defined codimension $d$ left ideal 
$$I/(\widehat{A}^{(m)}_+)^d\subset \widehat{A}^{(m)}/({\widehat{A}^{(m)}_+})^d,$$
%induced by mapping the class of a polynomial $P$ to $P(\varphi_1,\ldots,\varphi_m)v$,#
and thus a point in the Grassmannian $${\rm Gr}^d(\widehat{A}^{(m)}/(\widehat{A}^{(m)}_+)^d)$$ of $d$-dimensional quotients. By the description of ${}^0{\rm Hilb}^{(m)}(V)$ as a space of left ideals, this map is a closed immersion. In explicit coordinates, we can use the determinal semi-invariants $D_{P_1,\ldots,P_d}$ for classes $P_i\in A^{(m)}/(A^{(m)}_+)^d)$ as Pluecker coordinates of the above Grassmannian to embed $^0{\rm Hilb}^{(m)}(V)$ into a projective space.

%To see that it is a closed immersion, we work with explicit coordinates. From $(\varphi_1,\ldots,\varphi_m,v)$ we build a $d\times(m^d-1)/(m-1)$-matrix with columns $\varphi_{i_s}\ldots\varphi_{i_1}v$  indexed by all words $\omega=(i_1\ldots i_2)\in\Omega$ of length at most $d$. 

\begin{example}\label{example} Some small punctual noncommutative Hilbert schemes can be described explicitly.
\begin{itemize}
\item The punctual Hilbert schemes $^0{\rm Hilb}^{(1)}(V)$ all reduce to a single point, since a single nilpotent operator $\varphi$ admitting a cyclic vector is regular nilpotent. Equivalently, this point corresponds to the unique codimension $d$ ideal $(x^d)\in\mathbb{C}[[x]]$.
\item Trivially, ${}^0{\rm Hilb}^{(m)}(\mathbb{C})$ also reduces to a single point, corresponding to the augmentation ideal.
\item We have $$^0{\rm Hilb}^{(m)}(\mathbb{C}^2)\simeq\mathbb{P}^{m-1}.$$
Namely, any $(\varphi_1,\ldots,\varphi_m,v)$ can be represented by
$$\left[\begin{array}{ll}0&0\\ a_1&0\end{array}\right],\ldots,\left[\begin{array}{ll}0&0\\ a_1&0\end{array}\right],\left[{1\atop 0}\right],$$
unique up to rescaling the non-zero tuple $(a_1,\ldots,a_m)$. The corresponding left ideal is
$$(a_lx_k-a_kx_l,x_kx_l\, :\, k,l=1,\ldots,m).$$
\item Finally, we observe that $^0{\rm Hilb}^{(2)}(\mathbb{C}^3)$ is isomorphic to the cone over a  rational quartic scroll. Namely, we can embed this variety into a projective space with the coordinates $D_{P_1,P_2,P_3}$. Using a flag compatible with the operators $\varphi_1,\varphi_2$ as above, we easily see that $^0{\rm Hilb}^{(2)}(\mathbb{C}^3)$ already embeds into $\mathbb{P}^8$ using the functions
$$f_0=D_{1,x_1,x_2}\mbox{ and }f_{i,j,k}=D_{1,x_i,x_jx_k}\mbox{ for }i,j,k=1,2.$$
On the dense subset of $^0{\rm Hilb}^{(2)}(\mathbb{C}^3)$ of tuples
$$(\left[\begin{array}{lll}0&0&0\\ 1&0&0\\ 0&1&0\end{array}\right],\left[\begin{array}{lll}0&0&0\\ a&0&0\\ b&c&0\end{array}\right],\left[\begin{array}{l}1\\ 0\\ 0\end{array}\right]),$$
this embedding yields
$$(b:1:a:c:ac:a:a^2:ac:a^2)\in\mathbb{P}^8,$$
and we see that the image of this embedding is defined by the equalities
$$f_{1,1,2}=f_{2,1,1},\; f_{1,2,2}=f_{2,2,1}$$
together with the vanishing of all rank two minors of the matrix
$$\left[\begin{array}{llll} f_{1,1,1}&f_{1,1,2}&f_{1,2,1}&f_{1,2,2}\\ f_{2,1,1}&f_{2,1,2}&f_{2,2,1}&f_{2,2,2}\end{array}\right].$$
By \cite[Section 1.4]{Ha}, this defines the projective cone over the embedding of $\mathbb{P}^1\times\mathbb{P}^1$ into $\mathbb{P}^5$ via $\mathcal{O}_{\mathbb{P}^1\times\mathbb{P}^1}(1,2)$.
\end{itemize}
\end{example}

\section{Calculation of motives 1}\label{motive1}
%In the following, let $C(V)$ be either ${\rm End}(V)$ or $\mathcal{N}(V)$ (we can treat both cases in parallel); we also abbreviate $C(\mathbb{C}^d)$ as $C_d$.
In this section, we determine the motive of the $^0{\rm Hilb}^{(m)}(V)$ using a stratification of Harder-Narasimhan type \cite{RHNS}.\\[1ex]
We stratify $\mathcal{N}^{(m)}(V)\times V$ by the dimension of the subspace generated from the vector by the linear operators:
$$S_e(V)=\{(\varphi_*,v)\, :\, \dim\mathbb{C}\langle\varphi_1,\ldots,\varphi_k\rangle v=e\}.$$
The set $S_e(V)$ is the image of 
$$\widehat{S}_e(V)=\{(\varphi_*,v,U)\, :\, \mathbb{C}\langle\varphi_1,\ldots,\varphi_k\rangle v=U\}\subset \mathcal{N}^{(m)}(V)\times V\times{\rm Gr}_e(V)$$
under the projective map forgetting the subspace, which is bijective on points by definition. On the other hand, we have the projection $$p:\widehat{S}_e(V)\rightarrow{\rm Gr}_e(V),$$ which is equivariant for the natural ${\rm GL}(V)$-action, and thus turns $\widehat{S}_e(V)$ into a homogeneous bundle over ${\rm Gr}_e(V)$. To determine the fibre of this bundle, we fix a subspace $U\subset V$ of dimension $e$, choose a complement $W$, and represent vectors and linear operators with respect to the decomposition of $V=U\oplus W$. Then $p^{-1}(U)$ consists of all tuples

$$(\left[\begin{array}{ll}\varphi'_*&\zeta_*\\ 0&\varphi''_*\end{array}\right],\left[{v'\atop 0}\right])$$
for $(\varphi'_*,v)\in (\mathcal{N}^{(m)}(U)\times U)_{\rm st}$ and $(\varphi''_*)\in \mathcal{N}^{(m)}(W)$. This proves that
$$p^{-1}(U)\simeq S_e(U)\times \mathcal{N}^{(m)}(W)\times {\rm Hom}_\mathbb{C}(W,U)^m.$$
Denoting by $P(U)\subset{\rm GL}(V)$ the maximal parabolic of automorphisms fixing $U$, we thus conclude
$$\widehat{S}_e(V)\simeq{\rm GL}(V)\times^{P(U)}(S_e(U)\times \mathcal{N}^{(m)}(W)\times {\rm Hom}_\mathbb{C}(W,U)^m).$$
In the localized Grothendieck ring of complex varieties
$$R=K_0({\rm Var}_\mathbb{C})[\mathbb{L}^{-1},(1-\mathbb{L}^i)^{-1},\:\, i\geq 1],$$  we thus find an identity
$$[S_e(V)]=[\widehat{S}_e(V)]=\frac{[{\rm GL}(V)]}{[P(U)]}\cdot[S_e(U)]\cdot[\mathcal{N}^{(m)}(W)]\cdot[{\rm Hom}(W,U)^m].$$
This can be made explicit as
$$[S_e(V)]=\frac{[{\rm GL}_d(\mathbb{C})]}{[{\rm GL}_e(\mathbb{C})]\cdot[{\rm GL}_{d-e}(\mathbb{C})]}\cdot\mathbb{L}^{(m-1)e(d-e)}\cdot[\mathcal{N}^{(m)}(\mathbb{C}^{d-e})]\cdot[S_e(\mathbb{C}^e)].$$
Since the $S_e(V)$ stratify $\mathcal{N}^{(m)}(V)\times V$, we thus have

$$\mathbb{L}^{d}\cdot[\mathcal{N}^{(m)}(\mathbb{C}^d)]=\sum_{e=0}^d\frac{[{\rm GL}_d(\mathbb{C})]}{[{\rm GL}_e(\mathbb{C})]\cdot[{\rm GL}_{d-e}(\mathbb{C})]}\cdot\mathbb{L}^{(m-1)e(d-e)}\cdot[\mathcal{N}^{(m)}(\mathbb{C}^{d-e})]\cdot[S_e(\mathbb{C}^e)].$$

We have $$(\mathcal{N}^{(m)}(V)\times V)_{\rm st}=S_d(V),$$ and thus
$$\frac{\mathbb{L}^d\cdot [\mathcal{N}^{(m)}(\mathbb{C}^d)]}{[{\rm GL}_d(\mathbb{C})]}=\sum_{e=0}^d\mathbb{L}^{(m-1)e(d-e)}\cdot\frac{[\mathcal{N}^{(m)}(\mathbb{C}^{d-e})]}{[{\rm GL}_{d-e}(\mathbb{C})]}\cdot[^0{\rm Hilb}^{(m)}(\mathbb{C}^e)].$$

Now we form generating series, working in the (commutative) ring $R^{\rm tw}[[t]]$ with twisted multiplication 
$$t^e* t^f=\mathbb{L}^{(m-1)ef}t^{e+f}.$$
We find
$$\sum_{d\geq 0}\frac{[\mathcal{N}^{(m)}(\mathbb{C}^d)]}{[{\rm GL}_d(\mathbb{C})]}(\mathbb{L}t)^d=\sum_{d\geq 0}\frac{[\mathcal{N}^{(m)}(\mathbb{C}^d)]}{[{\rm GL}_d(\mathbb{C})]}t^d*\sum_{d\geq 0}[^0{\rm Hilb}^{(m)}(\mathbb{C}^d)]t^d.$$

Several formulas for the generating series of motives of nullcones are known \cite{GR}; we use \cite[Corollary 3.2]{GR}, which states that
$$\sum_{d\geq }\frac{[\mathcal{N}^{(m)}(\mathbb{C}^d)]}{[{\rm GL}_d(\mathbb{C})]}t^d*\sum_{d\geq 0}\frac{t^d}{(1-\mathbb{L})\cdot\ldots\cdot(1-\mathbb{L}^d)}=1.$$
Thus we find
$$\sum_{d\geq 0}\frac{t^d}{(1-\mathbb{L})\cdot\ldots\cdot(1-\mathbb{L}^d)}=\sum_{d\geq 0}\frac{(\mathbb{L}t)^d}{(1-\mathbb{L})\cdot\ldots\cdot(1-\mathbb{L}^d)}*\sum_{d\geq 0}[^0{\rm Hilb}^{(m)}(\mathbb{C}^d)]t^d.$$
Finally, we apply the $R$-linear map $$T(t^d)=\mathbb{L}^{-(m-1)d(d-1)/2}t^d,$$ which transforms the twisted multiplication into the usual one, resulting in 
$$\sum_{d\geq 0}\frac{\mathbb{L}^{-(m-1)d(d-1)/2}t^d}{(1-\mathbb{L})\cdot\ldots\cdot(1-\mathbb{L}^d)}=$$
$$=\sum_{d\geq 0}\frac{\mathbb{L}^{-(m-1)d(d-1)/2}(\mathbb{L}t)^d}{(1-\mathbb{L})\cdot\ldots\cdot(1-\mathbb{L}^d)}\cdot\sum_{d\geq 0}\mathbb{L}^{-(m-1)d(d-1)/2}\cdot[{}^0{\rm Hilb}^{(m)}(\mathbb{C}^d)]t^d$$
in $R[[t]]$. This proves the main result of this section:

\begin{theorem} Setting
$$H^{(m)}(q,t)=\sum_{d\geq 0}\frac{q^{-(m-1)d(d-1)/2}t^d}{(1-q)\cdot\ldots\cdot(1-q^d)}\in\mathbb{Q}(q)[[t]],$$
$$^0F^{(m)}(t)=\sum_{d\geq 0}\mathbb{L}^{-(m-1)d(d-1)/2}\cdot[^0{\rm Hilb}^{(m)}(\mathbb{C}^d)]t^d\in R[[t]],$$
we have
$$^0F^{(m)}(t)=\frac{H^{(m)}(\mathbb{L},t)}{H^{(m)}(\mathbb{L},\mathbb{L}t)}.$$
\end{theorem}

\begin{corollary}\label{corfunctional} The series $^0F^{(m)}(t)\in 1+tR[[t]]$ is determined by the functional equation
$$^0F^{(m)}(t)=1+t\cdot{\prod_{k=1}^m}{}^0F^{(m)}(\mathbb{L}^{1-k}t).$$
We have $$[{}^0{\rm Hilb}^{(m)}(\mathbb{C}^d)]={}^0h_d^{(m)}(\mathbb{L})\mbox{ and }[{\rm Hilb}^{(m)}(\mathbb{C}^d)]=h_d^{(m)}(\mathbb{L})$$
for polynomials $h_d^{(m)}(q),{}^0h_d^{(m)}(q)\in\mathbb{Z}[q]$ related by
$${}^0h_d^{(m)}(q)=q^{(m-1)d^2+d}\cdot h_d^{(m)}(q^{-1}).$$
\end{corollary}

\begin{proof} The $t^d$-coefficients $c_d(q)$ of $H^{(m)}(q,t)$ obviously satisfy
$$c_d(q)=\frac{q^{-(m-1)(d-1)}}{1-q^d}c_{d-1}(q),$$
and thus
$$H^{(m)}(q,t)-H^{(m)}(q,qt)=tH^{(m)}(q,q^{1-m}t),$$
which implies $$\frac{H^{(m)}(q,t)}{H^{(m)}(q,qt)}=1+t\cdot\frac{H^{(m)}(q,t)}{H^{(m)}(q,qt)}\cdot\ldots\cdot\frac{H^{(m)}(q,q^{1-m}t)}{H^{(m)}(q,q^{2-m}t)},$$
proving the first claim. Using this functional equation and the one in Theorem \ref{thmold}, we first  observe $\mathbb{L}$-polynomiality of the motives, and then the claimed relation between the polynomials by comparing coefficients.\end{proof}

\section{Construction of affine paving}\label{paving}

We recall the affine paving
$${\rm Hilb}^{(m)}(V)=\bigcup_TS_T,$$
the union ranging over all $d$-element trees $T\subset\Omega$. We will now show 
\begin{theorem}\label{thmpaving} The intersection of any $S_T$ with ${}^0{\rm Hilb}^{(m)}(V)$ is isomorphic to an affine space. Consequently, these intersections provide an affine paving of the variety ${}^0{\rm Hilb}^{(m)}(V)$.
\end{theorem}

\begin{proof} For words $\omega,\omega'\in\Omega$, write $\omega'\preceq\omega$ if $\omega'$ is a left subword of $\omega$. More precisely, we claim that $${}^0S_T=S_T\cap{}^0{\rm Hilb}^{(m)}(V)$$
consists of all $(\varphi_*,v)\in S_T$ such that
$$\varphi_k(e_\omega)\in\langle e_{\omega'}\, :\, \omega'\in T,\, \omega'<_{\rm lex}\omega k,\; \omega'\npreceq\omega\}$$ whenever $\omega\in T$, $\omega k\not\in T$. Denote by $S_T'$ the set of all such tuples $(\varphi_*,v)$. We first prove that $S_T'\subset {}^0S_T$. We define a new total order $\triangleleft$ on $T$ by the conditions
$$\omega\triangleleft\omega\omega'$$
and
$$\omega k\omega'\triangleleft\omega l\omega''$$
if $k>l$. The following relation between the orderings $<_{\rm lex}$, $\preceq$ and $\triangleleft$ is then immediate: if $\omega'<_{\rm lex}\omega k$ and $\omega'\npreceq \omega$, then $\omega\triangleleft\omega'$. We enumerate the tree 
$$T=\{\omega_1,\ldots,\omega_d\}$$
in the ordering $\triangleleft$ and consider the ordered basis $B$ with elements $e_l=e_{\omega_l}$ for $l=1,\ldots,d$ of $V$. For $(\varphi_*,v)\in S_T'$, by definition of this set, all operators $\varphi_k$ are represented by strictly lower triangular matrices with respect to $B$. This proves that $S_T'$ is contained in ${}^0{\rm Hilb}^{(m)}(V)$.\\[1ex]
Obviously $S_T'$ is closed in $S_T$, and thus 
$$S(V):=\bigcup_TS_T'$$
is a closed subset of ${}^0{\rm Hilb}^{(m)}(V)$; we denote by $Q(V)$ its complement. We will now prove by combinatorial means that, in the Grothendieck ring $K_0({\rm Var}_\mathbb{C})$, we have $[Q(V)]=0$. This implies that $Q(V)$ is empty, and thus that $S(V)$ already exhausts ${}^0{\rm Hilb}^{(m)}(V)$, proving the claim.\\[1ex]
For a tree $T$ as before, denote by $D(T)$ the set of all triples $$(\omega,k,\omega')\subset\Omega\times\{1,\ldots,m\}\times \Omega$$ such that $$\omega k\not\in T,\;\omega'\in T,\; \omega'<_{\rm lex}\omega k,\; \omega'\npreceq\omega.$$ Thus $|D(T)|$ is the dimension of $S_T'$. Every tree $T\not=\emptyset$ can be written uniquely in the form
$$T=\{\emptyset\}\cup\bigcup_{k=1}^mkT_k$$
for trees $T_1,\ldots,T_m$ (then $T$ is the grafting of $T_1,\ldots,T_m$ in the terminology of \cite[Section 5]{R}), and we have  $$|T|=\sum_{k=1}^m|T_k|+1.$$ It is then easy to see that the set $D(T)$ can be written aus the union of 
$$\{(k\omega,l,k\omega' ):\, 1\leq k\leq m,\, (\omega,l,\omega')\in D(T_k)\}$$
and
$$\{(k\omega,l,k'\omega' ):\, 1\leq k'<k\leq m,\, \omega\in T_k,\, \omega l\not\in T_k, \, \omega'\in T_{k'}\}.$$
The number of pairs $(\omega,k)\in \Omega\times\{1,\ldots,m\}$ such that $\omega\in T$, $\omega k\not\in T$ equals $(m-1)|T|+1$, thus we find
$$|D(T)|=\sum_k|D(T_k)|+\sum_{k'<k}|T_{k'}|((m-1)|T_k|+1).$$
We form the generating series of motives
$$F'(t)=\sum_{d\geq 0}\mathbb{L}^{-(m-1)d(d-1)/2}\cdot[S(\mathbb{C}^d)]t^d=\sum_{T}\mathbb{L}^{-(m-1)|T|(|T|-1)/2}\cdot\mathbb{L}^{|D(T)|}t^{|T|}.$$
Writing $T$ as the grafting of $T_1,\ldots,T_m$ as above, we can rewrite
$$F'(t)=1+\sum_{T_1,\ldots,T_m}\mathbb{L}^{f(T_1,\ldots,T_m)}t^{\sum_k|T_k|+1},$$
where $f(T_1,\ldots,T_m)$ is given as
$$-(m-1)(\sum_k|T_k|+1)\sum_k|T_k|/2+\sum_k|D(T_k)|+\sum_{k'<k}|T_{k'}|((m-1)|T_k|+1),$$
which easily simplifies to
$$-(m-1)\sum_k|T_k|(|T_k|-1)/2+\sum_k|D(T_k)|-\sum_k(k-1)|T_k|.$$
This implies
$$F'(t)=1+t\sum_{T_1,\ldots,T_m}\mathbb{L}^{-(m-1)\sum_k|T_k|(|T_k|-1)/2+\sum_k|D(T_k)|-\sum_k(k-1)|T_k|}t^{\sum_k|T_k|}=$$
$$=1+t\cdot\prod_{k=1}^mF'(\mathbb{L}^{1-k}t),$$
thus $F'(t)={}^0F(t)$ by Corollary \ref{corfunctional}.
Comparing coefficients, we find an equality of motives
$$[S(\mathbb{C}^d)]=[{}^0{\rm Hilb}^{(m)}(\mathbb{C}^d)],$$
as claimed.
\end{proof}

\section{Construction of resolution of singularities}\label{resolution}

Recall the Springer resolution of the variety of nilpotent linear operators, which consists of pairs of a nilpotent operator $\varphi$ on $V$ and a complete flag $0=F_0\subset F_1\subset\ldots\subset F_d=V$ which are compatible (in the sense that $\varphi(F_i)\subset F_{i-1}$ for all $i=1,\ldots,d$); this defines a homogeneous vector bundle over the variety ${\rm Fl}(V)$ of complete flags in $V$, with fibre isomorphic to $\mathfrak{n}(V)$, the space of operators compatible with a fixed flag $F_*^0$. We will imitate this construction in order to construct a resolution of singularities of ${}^0{\rm Hilb}^{(m)}(V)$.\\[1ex]
So define ${Y}^{(m)}(V)$ as the variety of tuples
$$((\varphi_*,v),F_*)\in(\mathcal{N}^{(m)}(V)\times V)_{\rm st}\times{\rm Fl}(V)$$
satisfying
$$\varphi_k(F_i)\subset F_{i-1}$$
for all $k=1,\ldots,m$, $i=1,\ldots,d$. Projection to $${\rm Fl}(V)\simeq{\rm GL}(V)/B(V)$$ (for $B(V)\subset{\rm GL}(V)$ the Borel subgroup fixing $F_*^0$) realizes $Y^{(m)}(V)$ as a homogeneous bundle with fibre
$$(\mathfrak{n}(V)^m\times V)_{\rm st}=(\mathfrak{n}(V)^m\times V)\cap(\mathcal{N}^{(m)}(V)\times V)_{\rm st}.$$
We can thus rewrite
$$Y^{(m)}(V)\simeq{\rm GL}(V)\times^{B(V)}(\mathfrak{n}(V)^m\times V)_{\rm st}.$$
Since $(\mathcal{N}^{(m)}(V)\times V)_{\rm st}$ admits a geometric ${\rm GL}(V)$-quotient, it also admits a geometric $B(V)$-quotient by \cite[(2.5)]{ES}. Thus its closed subset $(\mathfrak{n}(V)^m\times V)_{\rm st}$ also admits a $B(V)$-quotient, and we find
$$(\mathfrak{n}(V)^m\times V)_{\rm st}/B(V)\simeq({\rm GL}(V)\times^{B(V)}(\mathfrak{n}(V)^m\times V)_{\rm st})/{\rm GL}(V)\simeq Y^{(m)}(V)/{\rm GL}(V).$$
The map $$Y^{(m)}(V)\rightarrow (\mathcal{N}^{(m)}(V)\times V)_{\rm st}$$
forgetting the flag thus induces a projective map of quotients
$$\pi:Z^{(m)}(V):=Y^{(m)}(V)/{\rm GL}(V)\rightarrow{}^0{\rm Hilb}^{(m)}(V).$$
\begin{proposition} The map $\pi:Z^{(m)}(V)\rightarrow {}^0{\rm Hilb}^{(m)}(V)$ is a resolution of singularities.
\end{proposition}

\begin{proof} The map $\pi$ is surjective since $${\rm GL}(V)\mathfrak{n}(V)^m=\mathcal{N}^{(m)}(V).$$ The variety $Z^{(m)}(V)$ is irreducible and smooth since $\mathfrak{n}(V)^m\times V$ is so. Moreover, if $(\varphi_*)$ is a tuple of operators which induces non-zero tuples of maps $$(\overline{\varphi}_k:F^0_i/F^0_{i-1}\rightarrow F^0_{i-1}/F^0_{i-2})_k$$
for $i=2,\ldots,d$, then $F^0_*$ is the only flag compatible with all $\varphi_k$, and thus $\pi$ is bijective over the corresponding locus.\end{proof}
%$$\dim Z^{(m)}(V)=\dim Y^{(m)}(V)-\dim{\rm GL}(V)=m\dim \mathfrak{n}+\dim v-\dim{\rm GL}(V)=md(d-1)/2+d-d^2=(m-1)d(d-1)/2=\dim {\rm Hilb

\section{Calculation of motives 2}\label{motive2}

To calculate the motives of the $$Z^{(m)}(V)=(\mathfrak{n}(V)^m\times V)_{\rm st}/B(V),$$ we follow essentially the same strategy as in Section \ref{motive1}. We stratify $\mathfrak{n}(V)^m\times V$ by the dimension of the subspace generated from the vector by the linear operators:
$$P_e(V)=\{(\varphi_*,v)\, :\, \dim\mathbb{C}\langle\varphi_1,\ldots,\varphi_k\rangle v=e\}.$$
The set $P_e(V)$ is the image of 
$$\widehat{P}_e(V)=\{(\varphi_*,v,U)\, :\, \mathbb{C}\langle\varphi_1,\ldots,\varphi_k\rangle v=U\}\subset \mathfrak{n}(V)^m\times V\times{\rm Gr}_e(V)$$
under the projective map forgetting the subspace, which is bijective on points by definition. On the other hand, we have the projection $$p:\widehat{P}_e(V)\rightarrow{\rm Gr}_e(V),$$ which is equivariant for the natural ${B}(V)$-action. Under this action, ${\rm Gr}_e(V)$ decomposes into Schubert cells. Namely, we choose a basis $v_1,\ldots,v_d$ compatible with the flag $F_*^0$ in the sense that $F^0_i=\langle v_1,\ldots, v_i\rangle$ for all $i$. For an $e$-element subset $I=\{i_1<\ldots<i_e\}$ of $\{1,\ldots,d\}$, we denote by $U_I\subset V$ the subspace generated by the $v_i$ for $i\in I$. Then ${\rm Gr}_e(V)$ decomposes under $B(V)$ into orbits $\mathcal{O}_I=B(V)U_I$. We denote by $\widehat{P}_e(V)_I$ the inverse image under $p$ of $\mathcal{O}_I$, which is thus a homogeneous bundle over $\mathcal{O}_I$. To determine the fibre of this bundle, we consider the complement $W_I$ to $U_I$ generated by all $v_i$ for $$i\in\overline{I}=\{1,\ldots,e\}\setminus I=\{j_1<\ldots<j_{d-e}\},$$ and represent vectors and linear operators with respect to the ordered basis $$(v_{i_1},\ldots,v_{i_e},v_{j_1},\ldots,v_{j_{d-e}}).$$ Then $p^{-1}(U_I)$ consists of all tuples

$$(\left[\begin{array}{ll}\varphi'_*&\zeta_*\\ 0&\varphi''_*\end{array}\right],\left[{v'\atop 0}\right])$$
such $(\varphi'_*,v')\in (\mathfrak{n}(U_I)^m\times U_I)_{\rm st}$, $(\varphi''_*)\in \mathfrak{n}(W_I)^m$, $v'\in U_I$, and  $\zeta_1,\ldots,\zeta_m$ map each $v_i$ for $i\in\overline{I}$ into the span of the $v_j$ for $i<j\in I$. Denoting  $$\iota(I)=|\{((i,j)\, :\, \overline{I}\ni i<j\in I\}|,$$
this proves that
$$p^{-1}(U_I)\simeq P_e(U_I)\times \mathfrak{n}(W_I)^m\times\mathbb{C}^{m\iota(I)}.$$
Denoting by $P_I\subset{B}(V)$ the stabilizer of $U_I$, we can conclude
$$\widehat{P}_e(V)_I\simeq{B}(V)\times^{P_I}(P_e(U_I)\times \mathfrak{n}(W_I)^m\times \mathbb{C}^{m\iota(I)}).$$
The group $P_I$ has unipotent radical isomorphic to $\mathbb{C}^{\iota(I)}$ and Levi isomorphic to $B(U_I)\times B(W_I)$. In the localized Grothendieck ring $R$, we thus find an identity
$$[P_e(V)]=[\widehat{P}_e(V)]=\frac{[{B}(V)]}{[P_I]}\cdot[P_e(U_I)]\cdot[\mathfrak{n}(W_I)^m]\cdot\mathbb{L}^{m\iota(I)}=$$
$$=\mathbb{L}^{(m-1)\iota(I)}\cdot\frac{[{B}(V)]}{[B(U_I)]\cdot[B(W_I)]}\cdot[P_e(U_I)\cdot][\mathfrak{n}(W_I)^m].$$
Since
$$[B(\mathbb{C}^d)]=\mathbb{L}^{d(d-1)/2}\cdot(\mathbb{L}-1)^d,$$
this can be made explicit as
$$[P_e(V)]=\mathbb{L}^{(m-1)\iota(I)+e(d-e)}\cdot[P_e(\mathbb{C}^e)]\cdot[\mathfrak{n}(\mathbb{C}^{d-e})^m].$$
%\mathbb{\frac{[{\rm GL}_d(\mathbb{C})]}{[{\rm GL}_e(\mathbb{C})][{\rm GL}_{d-e}(\mathbb{C})]}\mathbb{L}^{(m-1)e(d-e)}[\mathcal{N}^{(m)}(\mathbb{C}^{d-e})][S_e(\mathbb{C}^e)].$$
Since the $P_e(V)$ stratify $\mathfrak{n}(V)^m\times V$, we thus have

$$\mathbb{L}^{d}\cdot[\mathfrak{n}(\mathbb{C}^d)^m]=\sum_{e=0}^d\sum_{|I|=e}\mathbb{L}^{(m-1)\iota(I)+e(d-e)}\cdot[P_e(\mathbb{C}^e)]\cdot[\mathfrak{n}(\mathbb{C}^{d-e})^m].$$
By the definition of $q$-binomial coefficients, we have
$$\sum_{|I|=e}\mathbb{L}^{(m-1)\iota(I)}=\left[{d\atop e}\right]_{\mathbb{L}^{m-1}},$$
thus the previous summation simplifies to
$$\mathbb{L}^{d}\cdot[\mathfrak{n}(\mathbb{C}^d)^m]=\sum_{e=0}^d\mathbb{L}^{e(d-e)}\cdot\left[{d\atop e}\right]_{\mathbb{L}^{m-1}}\cdot[P_e(\mathbb{C}^e)]\cdot[\mathfrak{n}(\mathbb{C}^{d-e})^m].$$

We have $$(\mathfrak{n}(V)^m\times V)_{\rm st}=P_d(V),$$ and thus
$$\frac{\mathbb{L}^d\cdot[\mathfrak{n}(\mathbb{C}^d)^m]}{[B(\mathbb{C}^d)]}=\sum_{e=0}^d\left[{d\atop e}\right]_{\mathbb{L}^{m-1}}\frac{[\mathfrak{n}(\mathbb{C}^{d-e})^m]}{[B(\mathbb{C}^{d-e})]}\cdot[{Z}^{(m)}(\mathbb{C}^e)].$$
Using the $q$-Pochhammer symbol
$$(a;q)_n=\prod_{k=0}^{n-1}(1-aq^k)$$
for $n\in\mathbb{N}\cup\{\infty\}$, the identity
$$\left[{d\atop e}\right]_q=\frac{(q;q)_d}{(q;q)_e(q;q)_{d-e}},$$
and the obvious equation $$[\mathfrak{n}(\mathbb{C}^d)]=\mathbb{L}^{d(d-1)/2},$$
we rewrite the above identity as
$$\frac{\mathbb{L}^d\cdot\mathbb{L}^{(m-1)d(d-1)/2}}{(\mathbb{L}^{m-1};\mathbb{L}^{m-1})_d}=\sum_{e=0}^d\frac{\mathbb{L}^{(m-1)(d-e)(d-e-1)/2}}{(\mathbb{L}^{m-1};\mathbb{L}^{m-1})_{d-e}}\cdot\frac{(\mathbb{L}-1)^e\cdot[Z^{(m)}(\mathbb{C}^e)]}{(\mathbb{L}^{m-1};\mathbb{L}^{m-1})_e}.$$
Defining
$$B(q,t)=\sum_{d\geq 0}\frac{q^{d(d-1)/2}t^d}{(q;q)_d},$$
we thus find an identity of generating functions
$$B(\mathbb{L}^{m-1},\mathbb{L}t)=B(\mathbb{L}^{m-1},t)\cdot\sum_{d\geq 0}\frac{[Z^{(m)}(\mathbb{C}^d)]}{(\mathbb{L}^{m-1},\mathbb{L}^{m-1})_d}((\mathbb{L}-1)t)^d.$$
One version of the classical $q$-binomial theorem \cite{A} reads
$$B(q,t)=(-t;q)_\infty,$$
and another one reads
$$\sum_{d\geq 0}\frac{(a;q)_d}{(q;q)_d}z^d=\frac{(az;q)_\infty}{(z;q)_\infty},$$
thus
$$\sum_{d\geq 0}\frac{[Z^{(m)}(\mathbb{C}^d)]}{(\mathbb{L}^{m-1},\mathbb{L}^{m-1})_d}((\mathbb{L}-1)t)^d=\frac{B(\mathbb{L}^{m-1},\mathbb{L}t)}{B(\mathbb{L}^{m-1},t)}=\sum_{d\geq 0}\frac{((-1)^d(\mathbb{L};\mathbb{L}^{m-1})_d}{(\mathbb{L}^{m-1};\mathbb{L}^{m-1})_d}t^d.$$
Comparing coefficients, we arrive at our main result.

\begin{theorem}\label{thmmotive2} For all $d$, we have 
$$[Z^{(m)}(\mathbb{C}^d)]=\frac{(\mathbb{L};\mathbb{L}^{m-1})_d}{(1-\mathbb{L}-1)}=\prod_{i=0}^{d-1}\frac{\mathbb{L}^{(m-1)i+1}-1}{\mathbb{L}-1}.$$
\end{theorem}

\section{Smallness}\label{small}

In this section, we prove:

\begin{theorem} The resolution
$$p:Z^{(m)}(V)\rightarrow{}^0{\rm Hilb}^{(m)}(V)$$
is small.
\end{theorem}

This means that the locus in the target where the fibre has at least dimension $r$ has codimension larger than $2r$, for all $r>0$. Equivalently, smallness of a resolution $f:X\rightarrow Y$ of a variety $Y$ by a proper map from a smooth irreducible variety $X$ is equivalent to $X\times_YX$ having dimension $\dim X$, with the diagonally embedded $X$ being the unique irreducible component of this dimension.\\[1ex]
So we consider the Steinberg-type variety $${\rm St}^{(m)}(V)=Z^{(m)}(V)\times_{{}^0{\rm Hilb}^{(m)}(V)}Z^{(m)}(V).$$
We will use the following compatibility of fibre products and geometric quotients, which should also hold without the assumption on $G$.

\begin{lemma} For $G$ a special algebraic group, $X$, $X'$ and $X''$ being $G$-varieties admitting geometric quotients, and a diagram of $G$-equivariant maps $X\rightarrow X''\leftarrow X'$, we have an isomorphism
$$X/G\times_{X''/G}X'/G\simeq (X\times_{X''}X')/G.$$
\end{lemma}

\begin{proof} The universal property of fibre products provides a map from the left hand side to the right hand side. Whether this is an isomorphism can be verified locally. Since $G$ is special, all geometric quotients are Zariski-locally trivial, and we can assume $X=G\times F$ with $G$-action on the left factor, and similarly for $X'$ and $X''$. Then the claimed isomorphism is obvious.
\end{proof}

Using this lemma, we can equivalently write ${\rm St}^{(m)}(V)$ as the quotient by the natural ${\rm GL}(V)$-action of
$$Y^{(m)}(V)\times_{(\mathcal{N}^{(m)}(V)\times V)_{\rm st}}Y^{(m)}(V),$$
this variety being isomorphic to the variety of tuples
$$Y_2^{(m)}(V)=((\varphi_*,v),F_*,F_*')\in(\mathcal{N}^{(m)}(V)\times V)_{\rm st}\times{\rm Fl}(V)\times{\rm Fl}(V)$$
such that $$\varphi_k(F_i)\subset F_{i-1},\; \varphi_k(F_i')\subset F_{i-1}'$$
for all $k=1,\ldots,m$, $i=1,\ldots,d$. We consider the projection
$$p_2:Y_2^{(m)}(V)\rightarrow{\rm Fl}(V)\times{\rm Fl}(V),$$
which is ${\rm GL}(V)$-equivariant. The orbits $\mathcal{O}_\sigma$ under the diagonal action in the target are parametrized by permutations, with $\mathcal{O}_\sigma$ for $\sigma\in S_d$ the orbit of the pair of flags $(F_\sigma,F'_\sigma)$ defined by $$(F_\sigma)_i=\langle v_1,\ldots,v_i\rangle,\; (F_\sigma')_i=\langle v_{\sigma 1},\ldots,v_{\sigma i}\rangle$$
for all $i$. If the inverse image $p_2^{-1}(\mathcal{O}_\sigma)$ is non-empty, it is thus a homogeneous bundle over $\mathcal{O}_\sigma$. The fibre over $(F_\sigma,F_\sigma')$ consists of tuples $(\varphi_1,\ldots,\varphi_m,v)$ in $(\mathcal{N}^{(m)}(V)\times V)_{\rm st}$ where the $\varphi_k$ map each $v_i$ to a linear combination of the $v_j$ such that $j>i$ and $\sigma j>\sigma i$, thus
$$\dim p^{-1}(F_\sigma,F'_\sigma)=m(d(d-1)/2-l(\sigma))+d.$$
Similarly, the stabilizer of $(F_\sigma,F_\sigma')$ in ${\rm GL}(V)$ consists of invertible maps mapping each $v_i$ to a linear combination of the $v_j$ such that $j\geq i$ and $\sigma j\geq\sigma i$, and thus
$$\dim \mathcal{O}_\sigma=d(d-1)/2+l(\sigma).$$ Thus $p_2^{-1}(\mathcal{O}_\sigma)$ is irreducible of dimension
$$\dim p_2^{-1}(\mathcal{O}_\sigma)=(m-1)(d(d-1)/2-l(\sigma))+d^2.$$
We thus see that $Y_2^{(m)}(V)$ has dimension $(m-1)d(d-1)/2+d^2$, and $p_2^{-1}(\mathcal{O}_{\rm id})$ is the unique irreducible component of this dimension. Consequently,
$${\rm St}^{(m)}(V)=Y_2^{(m)}(V)/{\rm GL}(V)$$
has dimension $m(m-1)/2$, with a unique irreducible component of this dimension. This proves the theorem.

\begin{corollary}\label{coric} The Poincar\'e polynomial in rational intersection homology of ${}^0{\rm Hilb}^{(m)}(V)$ is given by
$$\sum_i\dim{\rm IH}^i({}^0{\rm Hilb}^{(m)}(V),\mathbb{Q})q^{i/2}=\prod_{i=0}^{d-1}\frac{q^{(m-1)i+1}-1}{q-1}.$$
\end{corollary}

\begin{proof} The intersection homology of a variety equals the singular cohomology of a small resolution. Since the motive of the resolution $Z^{(m)}(V)$ is a polynomial in $\mathbb{L}$ by Theorem \ref{thmmotive2}, this polynomial evaluated at $q$ also gives the Poincar\'e polynomial in singular cohomology.
\end{proof}

We conclude our study of the punctual noncommutative Hilbert schemes and their resolutions with a  conjecture.

\begin{conjecture} There exist affine pavings of ${}^0{\rm Hilb}^{(m)}(V)$ and its resolution $Z^{(m)}(V)$, with affine parts $S_T^{\rm new}$ indexed by $m$-ary trees with $d$ nodes (respectively with affine parts $S_{\widehat{T}}$ indexed by pairs $\widehat{T}=(T,f)$ consisting of an $m$-ary tree $T$ with $d$-nodes and a compatible ordering $f:T\rightarrow \{1,\ldots,d\}$ in the sense that $f(\omega)\leq f(\omega')$ if $\omega\preceq\omega'$), such that the resolution map $\pi$ maps affine pieces to affine pieces, and the fibre over a point in $S_T^{\rm new}$ admits an affine paving by the $S_{(T,f)}$ for $f$ compatible with $T$.
\end{conjecture}

We note that an affine paving different from the one constructed in Section \ref{paving} is necessary for proving this conjecture. For example, the affine piece $S_T^{\rm new}\subset{}^0{\rm Hilb}^{(2)}(\mathbb{C}^3)$ for $T=\{\emptyset,(1),(2)\}$ should consist only of the isolated singularity, which is resolved to a projective line consisting of two affine pieces for the two compatible orderings.

\end{document}